\documentclass{amsart}

      \usepackage{amssymb}
\usepackage[T1]{fontenc}
\usepackage[latin9]{inputenc}
\usepackage{amsmath}
\usepackage{amssymb}
\usepackage[all]{xy}
      \theoremstyle{plain}
      
      \newtheorem*{thma}{Theorem}
      \newtheorem*{lemma}{Lemma}

      \theoremstyle{definition}

      \theoremstyle{remark}

      \makeatletter
      \def\@setcopyright{}
      \def\serieslogo@{}
      \makeatother
   
\begin{document}

%



   \author{Mehdi Khorami and Mark Mahowald}




   \title{A filtration of unoriented cobordism}


   \begin{abstract}
     We give a filtration of the unoriented cobordism ring 
      using the infinite symplectic group, with polynomial generators given one at a time. The generating manifolds are also constructed using the cup construction. 
     \end{abstract}




   \date{\today}


   \maketitle



In this paper we give a homotopy filtration of the unoriented cobordism $MO$, with the polynomial generators constructed and given one at a time. Such filtration arises from the corresponding filtration given for $\text BO$. The filtration is indexed on triples of integers $(n, j, i)$ with $n\geq 1$ and $i, j\geq 0$ (with the exception that when $n=1$, $j\geq 1$). We order such triples in the following way: $(n, j, i)<(n', j', i')$ when 

\begin{tabular}{l r}
               $n < n'$, &  or\\

$n=n'$ and $j< j'$, & or  \\
$n=n'$, $j=j'$ and $i< i'$.&   \\

\end{tabular}

Our main theorem is the following. Mod 2 coefficients are assumed throughout.

\begin {thma} There exist a filtration $F_{(n,j,i)}$ of $\text{BO}$ such that if $MF_{(n,j,i)}$ is the Thom complex of the inclusion $F_{(n,j,i)}\to \text{BO}$, then $\pi_*(MF_{(n,j,i)})/\pi_*(MF_{<(n,j,i)})$ is a polynomial algebra $ \mathbb Z/2\mathbb Z[x]$ on one generator.
\end{thma}
  Recall \cite{stong} that the unoriented cobordism ring $\pi_*(\text{MO})=\mathbb Z/2\mathbb Z[x_2, x_4, x_5, \cdots] $ is a polynomial algebra with one generator $x_n$ in degree $n$ for each $n$ not of the form $2^i-1$ for some $i>0$. It is well known that $\mathbb RP^{2n}$ can be taken to be the polynomial generator of degree $2n$ of the cobordism ring. In \cite{dold}, Dold constructs manifolds that represent the odd dimensional polynomial generators of $MO_*$.

  We are interested in mapping certain H-spaces into $\text{BO}$, where the homology of the associated Thom complexes include the dual Steenrod algebra $\mathcal A_*$ as a tensor product. The key here is that the Adams spectral  sequence collapses for such spaces. One such space is $\Omega ^2 S^3$. Recall that 
  \[H_*(\Omega ^2 S^3)\cong  \mathbb Z/2\mathbb Z[\xi_1,\xi_2, \cdots]\cong \mathcal A_*\]  where $|\xi_i|= 2^i-1$. It is an observation of Mahowald that there exist a map $\Omega ^2 S^3\to \text{BO}$ whose Thom complex is the Eilenberg-Mac Lane spectrum $K(\mathbb Z/2\mathbb Z)$. Let us make this point precise.

  Let $\eta:S^1\to \text{BO}\in \pi_1(\text{BO})$ be the generator. Since $\text{BO}$ is a double loop space, $\text{BO}\simeq \Omega ^2 X$ for some $X$. Let $\gamma$ be the composite
  \[ \Omega^2  S^3=\Omega^2 \Sigma  ^2 S^1\stackrel{\Omega^2 \Sigma  ^2 \eta}\longrightarrow  {\Omega^2 \Sigma  ^2 BO}\simeq \Omega^2 \Sigma  ^2 \Omega^2 X \to \Omega ^2 X\simeq \text{BO}\] 
Mahowald has shown in \cite{mah} that if $M(\gamma)$ is the Thom complex associated with $\gamma$, then 
$M(\gamma)\cong K(\mathbb Z/2\mathbb Z)$. 

Recall that $\mathbb Z\times BSp$ appears as the fourth space in the real $K$-theory spectrum $KO$, that is $\Omega^4 BSp\simeq \text{BO} \times \mathbb Z$, and so $\Omega^3 Sp\simeq \text{BO} \times \mathbb Z$. Here $Sp= \varinjlim Sp(n)$ is the infinite symplectic group. We will show how the homology of Thom complexes of $\Omega^3  Sp(n)$ includes $\mathcal A_*$ as a tensor product, and thus we are lead to considering the maps $\Omega^3 Sp(n)\to \text{BO}$, and in fact this is where the filtration in the above theorem arises.

Let us start by recalling some classical definitions and results. 

Let $A$ be an algebra over a ring $R$. Then elements $x_1, x_2, \cdots x_n, \cdots $ in $A$ are said to form  a \emph{simple system of generators} for $A$ if the monomials $x_1^{\epsilon_i}x_2^{\epsilon_i}\cdots x_m^{\epsilon_i}$ ($\epsilon_i=0$ or 1, $m\geq 0$) form a basis for $A$ over $R$. For example in an exterior algebra $E(x_1, x_2, \cdots)$, the elements $x_1, x_2, \cdots$ form a simple system of generators. In the polynomial ring $R[x]$, the monomials $x, x^2, x^4, \cdots , x^{2^n}, \cdots $ form a simple system of generators. 

For a path fibration $\Omega B\to PB\to B$ over a space $B$, let $\sigma': \tilde H_{n-1}(\Omega B)\to \tilde H_n(B)$ be the homology suspension.
\begin {thma}(A. Borel) Let $B$ be simply connected H-group. Let $a_1, a_2, \cdots$ element of $H_*(\Omega B, R) $ such that for each $n$ only finitely many $a_i$'s lie in $H_*(\Omega B, R) $ and $\sigma'(a_1), \sigma'(a_1)\cdots$ form a simple system of generators for the Pontrjagin ring $H_*( B, R) $. Then $H_*(\Omega B, R) \cong R[a_1,a_2, \cdots]$. 
\end {thma}
\begin{proof} See \cite{switzer}, theorem 15.60.
\end {proof}

Using the fibration $Sp(n-1)\to Sp(n)\to S^{4n-1}$, one can easily see that:
\begin{lemma}  The homology of $Sp(n)$ is an exterior algebra on generators in degrees of the form $4i-1$ for $1\leq i \leq n$, that is \[H_*(Sp(n))=E(a_{4i-1}, 1\leq i \leq n).\]

\end{lemma}
For details see  \cite{hatcher}, corollary 4D.3.

Borel's theorem then implies that the homology of the loop space $\Omega Sp(n)$ is a polynomial algebra on generators in degrees $4i-2$: 
 \[H_*(\Omega Sp(n))=\mathbb Z/2\mathbb Z[b_{4i-2}, 1\leq i \leq n]\]
 
 For the homology of the double loop space $\Omega ^2Sp(n)$, we choose a simple system of generators for the above polynomial algebra (i.e. $b_{4i-2}^{2^j}$ for 1$\leq i \leq n$ and $j\geq 0$) and $H_*(\Omega^2 Sp(n))$ will then be a polynomial algebra on generators 
 
 \[\{\sigma'(b_{4i-1}), \sigma'(b_{4i-1}^2),  \sigma'(b_{4i-1}^4),\cdots, 1\leq i \leq n\}\]
 
We are going to consider the Thom complexes of the maps $\Omega^3 (Sp(n))\to \text{BO}$. For each $n$ we obtain an infinite family of polynomial generators given by such complexes. 
 We start with $n=1$.
 
 \vspace{.2in}
\noindent\textbf{\boldmath{Sp(1)}}. It is easily seen that $Sp(1)=S^3$, and so we are looking at the map $\Omega^3 S^3\to \text{BO}$. The homology of $\Omega S^3$ is a polynomial algebra on one generator in degree 2 and $\Omega S^3$ has a CW structure 
\[\Omega S^3=J(S^2)\simeq S^2\cup  e^4\cup e^6\cup \cdots \]
 with all the cells attached nontrivially. Here $J(S^2)$ denotes the James reduced product on $S^2$. Now we consider pieces of this product starting with $\Omega^2 (S^2\cup e^4\cup e^6)$. Let $\doteq$ denote homological equivalence. Since the 6-cell $e^6$ is the product of the cells $e^2$ and $e^4$, we have  \[\Omega ^2(S^2\cup e^4\cup e^6)\doteq \Omega ^2 S^2 \times \Omega^2 S^4\] 
 
 The homology of $\Omega ^2 S^2$ is $\mathbb Z$ in degree zero and is equal to the homology of $\Omega ^2 S^3$ in higher dimensions (this can easily be seen from the Hopf fibration $S^1\to S^3\to S^2$). In fact, $\Omega ^2 S^3$ is homotopy equivalent to the connected component of $\Omega ^2 S^2$ containing the base point. It is worth noting that the map $\gamma$ introduced above is the composite map
 \[\Omega ^2 S^3\hookrightarrow \Omega ^2 S^2\to \Omega^3 S^3\to \text{BO}\]
 and the $0^{\text{th}}$ space in our filtration $F_{(1,0,0)}$ is the image of this map in $\text{BO}$.
 
 For $\Omega ^2 S^4$, we can write 
 \[\Omega ^2 S^4\simeq \Omega \Omega \Sigma S^3\simeq \Omega J(S^3)\] We are now going to consider the pieces $J_{2^i-1}(S^3)$, for various $i$, starting with $J_0(S^3)=S^3$. Let $\gamma_2$ be the map $\Omega ^2 S^2\times \Omega J_0(S^3)\to \text{BO}$, and let $M(\gamma_2)$ be the associated Thom complex. The homology of this complex is given by 
 \[H_*(M(\gamma_2))\cong \mathcal A_*\otimes H_*(\Omega S^3)\] the best possible outcome, for feeding this into the Adams spectral sequence we get 
 \[E_2^{s,t}=\text{Ext}_{\mathcal A_*}^{s,t}(\mathbb Z/2\mathbb Z, \mathcal A_*\otimes H_*(\Omega S^3))=\text{Ext}_{\mathbb Z/2\mathbb Z}^{s,t}(\mathbb Z/2\mathbb Z, H_*(\Omega S^3))\]\[=
  \left\{
	\begin{array}{ll}
		\text{Hom}_{\mathbb Z/2\mathbb Z}^t(\mathbb Z/2\mathbb Z, H_*(\Omega S^3)) &  s = 0 \\
		0 &  s > 0
	\end{array}
\right. = 
 \left\{
	\begin{array}{ll}
		H_t(\Omega S^3)) & s = 0 \\
		0 &  s > 0
	\end{array}
\right.
\]
Hence 
\[ \pi_*(M(\gamma_2))\cong H_*(\Omega S^3))=\mathbb Z/2\mathbb Z[x_2]\] 
Thus we obtain our first polynomial generator $x_2$ in degree 2. 

Next we consider the Thom complex $M(\gamma_5)$ of the map
 \[\gamma_5: \Omega ^2 S^2\times \Omega J_1(S^3)\to \text{BO}\]
 The argument above repeats to give that 
\[\pi_*(M(\gamma_5))\cong H_*(\Omega J_1(S^3))\cong \mathbb Z/2\mathbb Z[x_2, x_5] \] giving rise to a new generator in degree $2\cdot3-1=5$ denoted by $x_5$. 

We can repeat this process by forming the Thom complexes of the maps  $$\gamma_{2^i-1}: \Omega ^2 S^2\times \Omega J_{2^i-1}(S^3)\to \text{BO}$$ giving rise to a family of generators in degrees $3\cdot2^i-1$, for all nonnegative $i$. In addition, we take $F_{(1,1,i)}$ to be the image of $\gamma_{2^i-1}$ in $\text{BO}$.

 We now construct manifold generators representing this family. Let $X$ be any space. Following \cite{RLB}, let \emph{cup ``m" construction} $P(m, X)$ on $X$ be the quotient of $S^m\times X\times X$ by the relations $(u, x, y)\sim (-u, y, x)$. If $X$ is an $n$-manifold, then $P(m, X)$ is an $(m+2n)$-manifold. We use this construction for $m=1, 2$ only. Proposition 4.1 of \cite{RLB} implies that if a manifold $M$ represent an indecomposable cobordism class, then $P(1, M)$ is also indecomposable. In addition, if $M$ is of even dimension, then $P(2, M)$ is also indecomposable.

We construct our manifolds recursively. Let $M_2$ be the real projective space $\mathbb RP^2$, representing $x_2$. By the paragraph above, $M_5=P(1, M_2)$ is a  manifold of dimension 5 which represents an indecomposable cobordism class, and can be taken to be $x_5$. If $M_{3\cdot 2^i-1}$ is constructed, we can set $M_{3\cdot 2^{i+1}-1}=P(1, M_{3\cdot 2^i-1})$ to get a manifold of dimension $2(3\cdot 2^i-1)+1=3\cdot 2^{i+1}-1$, representing $x_{3\cdot 2^{i+1}-1}$.

The next step is to consider the next bit of the product, that is, we consider $\Omega ^2J_{2^3-1}(S^2)$, and homologicaly 
 \[\Omega ^2J_7(S^2)\doteq \Omega ^2 S^2 \times \Omega^2 S^4\times \Omega^2 S^8\] 
Similar to the above case, $\Omega^2 S^8=\Omega\Omega \Sigma S^7=\Omega J(S^7)$, and we can consider the maps $\Omega ^2 S^2 \times \Omega^2 S^4\times \Omega J_{2^i-1}( S^7)\to\text{BO}$ for various $i$. The resulting Thom complexes will give new polynomial generators in degrees $7\cdot 2^i-1$, in addition to the generators arising from $\Omega ^2 S^2 \times \Omega^2 S^4$. 

We use cup 2 and then cup 1 construction to construct manifolds representing this family of generators. The lowest degree generator is in degree 6, and so we can start with $M_6=\mathbb P(2, M_2)$, and if $M_{7\cdot 2^i-1}$ is constructed, we set $M_{7\cdot2^{i+1}-1}=P(1, M_{7\cdot 2^i-1})$. The dimension of $M_{7\cdot 2^{i+1}-1}$ is then $2(7~\cdot2^i-1)+1=7\cdot 2^{i+1}-1$, and can be taken to represent $x_{7\cdot 2^{i+1}-1}$.

It should now be clear that the piece $\Omega ^2J_{2^j-1}(S^2)$, for a fixed $j\geq 1$, will give infinitely many generators in degrees $(2^{j+1}-1)2^i-1$ for all nonnegative $i$. The lowest degree generator will be in dimension $2^{j+1}-2$, which can be constructed from the generator in degree $2^j-2$ by cup 2 construction starting with $\mathbb RP^2$. The higher degree generators in the family are then given using cup 1 construction.  This shows that for $n=1$ all the generators can be built from $\mathbb RP^2$ recursively.

\vspace{.2in}

\noindent\textbf{\boldmath{Sp(2)}}.  The homology ring $H_*( Sp(2))=E(a_3,a_7)$ has a generator in degree 7 and $\frac{Sp(2)}{Sp(1)}=S^7$. Again homologicaly we have 
\[\Omega S^7\doteq \prod_{j\geq 0} S^{6\cdot2^j}\]
and so 
\[\Omega^3 S^7\doteq \Omega^2 \prod S^{6\cdot2^j}\doteq \Omega \prod \Omega \Sigma S^{6\cdot2^j-1}\doteq \Omega \prod J(S^{6\cdot2^j-1})\]
As before, starting with $S^5$, we consider the Thom complexes of the maps 
\[\Omega^3 S^3\times \Omega J_{2^i-1} S^5\to \text{BO}\]
The homology of such complexes will include the dual Steenrod algebra $\mathcal A_*$ as a tensor product, resulting from $\Omega^3 S^3$. Thus in addition to the generators obtained from ${Sp}(1)$, we get polynomial generators in degrees $5\cdot 2^i-1$, $i\geq 0$. This is the first family of generators arising from $\text{Sp}(2)$, with the first one in degree 4. The representing manifolds can then be constructed recursively, starting from $\mathbb RP^4$ and using cup 1 construction, analogous to the case of ${Sp}(1)$.

Similarly, $S^{11}$ gives generators in degrees $11\cdot 2^i-1$, with the first one in degree 10.  For this family, we start with $M_{10}=\mathbb P(2, \mathbb RP^4)$ and apply cup 1 construction to get the higher degree generators. Similarly for $S^{23}, S^{47},\cdots$. This shows that  $\text{Sp}(2)$ gives rise to generators in degrees $(6\cdot 2^j-1)2^i-1$ for all nonnegative $i$ and $j$, all of which can be constructed from $\mathbb RP^4$.
\vspace{.2in}

\noindent\textbf{\boldmath{Sp(n)}}. Quite generally, assume that polynomial generators are constructed for the map $\Omega ^3\text{Sp}(n-1)\to \text {BO}$. $Sp(n)$ will then give us a new exterior generator in degree $4n-1$, and $\frac{Sp(n)}{Sp(n-1)}=S^{4n-1}$.  Homologicaly we have 
\[\Omega S^{4n-1}\doteq \prod_{j\geq 0} S^{(4n-2)2^j}\]
and so 
\[\Omega^3 S^{4n-1}\doteq \Omega^2 \prod S^{(4n-2)2^j}\doteq \Omega \prod \Omega \Sigma S^{(4n-2)2^j-1}\doteq \Omega \prod J(S^{(4n-2)2^j-1})\]
This shows that the various pieces $J_{2^i-1}(S^{(4n-2)2^j-1})$ will give rise to generators in degrees $((4n-2)2^j-1)2^i-1$ for $n\geq 1$ and all nonnegative $i$ and $j$ (when $n=1$,  $j\geq 1$).  $F_{(n, j, i)}$ is then the image of $Sp(n-1)\times \Omega J_{2^i-1}(S^{(4n-2)2^j-1})$ under the map 
\[Sp(n-1)\times \Omega J_{2^i-1}(S^{(4n-2)2^j-1})\to \Omega ^3\text{Sp}(n)\to \text {BO}\]

Thus for fixed $n$ and $j$, we obtain an infinite family of generators with the first one in degree $(4n-2)2^j-2$ for $i=0$.  Manifolds representing such family are constructed using the cup construction, in the following manner. For a fixed $n$, we have infinite family of generators corresponding to each $j=0, 1, 2, \cdots$. When $j=0$, the family of generators obtained by varying $i$ can be constructed using cup 1 construction recursively starting with $\mathbb RP^{4(n-1)}$. For infinite family corresponding to $j=1$, we start with cup 2 construction on $\mathbb RP^{4(n-1)}$, and then cup 1 for the rest of the generators in the family. The lowest degree generator in the family corresponding to $j=2$ is obtained from cup 2 construction on lowest degree generator in the family corresponding to $j=1$, and then cup 1 for the rest of the family. And so on. The only exception is when $n=1$, in which case we start with $\mathbb RP^2$, as demonstrated above for $Sp(1)$.

Now it's easily seen that the numbers $((4n-2)2^j-1)2^i-1$ are not of the form $2^k-1$ for any $k$, and any integer not of the form $2^k-1$ can be written uniquely in the form $((4n-2)2^j-1)2^i-1$. Thus we obtain all the polynomial generators of $MO_*$, all of which can be constructed using cup 1 and cup 2 constructions from $\mathbb RP^2$ and $\mathbb RP^{4n}$ for $n\geq1$.


\end{document}